\documentclass[11pt]{amsproc}

\usepackage[hiresbb]{graphicx}
\usepackage{xcolor}

\usepackage{amsmath,amsthm,amssymb,mathtools}
\allowdisplaybreaks

\usepackage{fullpage}

\usepackage{here}
\usepackage{hyperref}
\usepackage[abbrev]{amsrefs}
\renewcommand{\MR}[1]{}

\usepackage{tikz}
\usepackage{float}
\usetikzlibrary{arrows.meta,decorations.markings}

\usepackage{aliascnt}
\usepackage[capitalize,nameinlink,noabbrev,nosort]{cleveref}

\hypersetup{
  colorlinks=true,
  linkcolor=brown,
  citecolor=brown,
  filecolor=brown,
  urlcolor=brown,
  pdftitle={Modular integrals for individual modular knots},
  pdfauthor={Toshiki Matsusaka},
  pdfsubject={2020 Mathematics Subject Classification: 11F67, 11F20, 57K10},
}

\makeatletter
\@namedef{subjclassname@2020}{\textup{2020} Mathematics Subject Classification}
\makeatother

\theoremstyle{remark}

\newaliascnt{remark}{theoremcounter}

\aliascntresetthe{remark}

\theoremstyle{definition}

\newaliascnt{definition}{theoremcounter}
\newtheorem{definition}[definition]{Definition}
\aliascntresetthe{definition}

\newaliascnt{example}{theoremcounter}

\aliascntresetthe{example}

\theoremstyle{plain}

\newaliascnt{lemma}{theoremcounter}
\newtheorem{lemma}[lemma]{Lemma}
\aliascntresetthe{lemma}

\newaliascnt{proposition}{theoremcounter}
\newtheorem{proposition}[proposition]{Proposition}
\aliascntresetthe{proposition}

\newaliascnt{corollary}{theoremcounter}

\aliascntresetthe{corollary}

\newaliascnt{conjecture}{theoremcounter}

\aliascntresetthe{conjecture}

\newaliascnt{theorem}{theoremcounter}
\newtheorem{theorem}[theorem]{Theorem}
\aliascntresetthe{theorem}

\newaliascnt{question}{theoremcounter}

\aliascntresetthe{question}

\newaliascnt{exercise}{theoremcounter}

\aliascntresetthe{exercise}

\numberwithin{equation}{section}

\newcommand{\Z}{\mathbb{Z}}
\newcommand{\Q}{\mathbb{Q}}

\providecommand{\C}{}
\renewcommand{\C}{\mathbb{C}}

\newcommand{\dd}{\mathrm{d}}

\ExplSyntaxOn
\clist_map_inline:nn
  { A,B,C,D,E,F,G,H,I,J,K,L,M,N,O,P,Q,R,S,T,U,V,W,X,Y,Z }
  {
    \cs_new:cpn { cal#1 } { \mathcal{#1} }
  }
\ExplSyntaxOff

\ExplSyntaxOn
\clist_map_inline:nn
  { A,B,C,D,E,F,G,H,I,J,K,L,M,N,O,P,Q,R,S,T,U,V,W,X,Y,Z }
  {
    \cs_new:cpn { mat#1 } { \mathrm{#1} }
  }
\ExplSyntaxOff

\DeclareMathOperator{\ReNew}{Re}
\renewcommand{\Re}{\ReNew}

\def\sts#1#2{\genfrac{\{}{\}}{0pt}{}{#1}{#2}}

\begin{document}

\title[]{Poly-Bernoulli numbers from shifted log-sine integrals}

\author[]{Toshiki Matsusaka}
\address{Faculty of Mathematics, Kyushu University, Motooka 744, Nishi-ku, Fukuoka 819-0395, Japan}
\email{matsusaka@math.kyushu-u.ac.jp}


\keywords{}
\subjclass[2020]{11B68, 11M32}



\maketitle

\begin{abstract}
	In 1999, Arakawa and Kaneko introduced a zeta function connecting multiple zeta values with poly-Bernoulli numbers through its special values. Log-sine integrals are also known to be related to zeta values at positive integers. In this article, we study the values of a shifted log-sine integral at negative integers and establish a connection with anti-diagonal sums of poly-Bernoulli numbers. This gives a refinement of Stephan's observation on the central binomial series.
\end{abstract}



\section{Introduction}

The \emph{poly-Bernoulli numbers}, introduced by Kaneko~\cite{Kaneko1997}, generalize the classical Bernoulli numbers. For an integer $k$, they are defined by the generating function
\[
	\sum_{n=0}^\infty B_n^{(k)} \frac{x^n}{n!} = \frac{\mathrm{Li}_k(1-e^{-x})}{1-e^{-x}},
\]
where $\mathrm{Li}_k(z) = \sum_{m=1}^\infty z^m/m^k$ is the polylogarithm. Motivated by Kaneko's work, Arakawa and Kaneko~\cite{ArakawaKaneko1999} introduced a zeta function whose values at positive integers are related to multiple zeta values, while its values at negative integers are connected with the poly-Bernoulli numbers. In this respect, the situation resembles Euler's classical results for the Riemann zeta function.

A related connection appears in the central binomial series
\[
    \sum_{m=1}^\infty \frac{1}{\binom{2m}{m}m^s},
\]
whose values at negative integers belong to $\Q + \Q\pi/\sqrt{3}$ by Lehmer's formulas~\cite{Lehmer1985}. Stephan observed that the rational part of each such value is one third of an anti-diagonal sum of poly-Bernoulli numbers with negative index (in OEIS, A098830). This observation, discussed by Kaneko~\cite{Kaneko2008}, was proved by B\'{e}nyi and the author~\cite{BenyiMatsusaka2023}.

In this article, we introduce a shifted log-sine integral $\mathrm{SLs}(s;\sigma)$ and show that, at $\sigma=\pi/3$, its analytically continued values at negative integers coincide with these rational parts. We also obtain an analogous formula for variable $\sigma$ using a polynomial generalization of the poly-Bernoulli numbers.

\begin{definition}
	For $0 < \sigma \le \pi$ and $s \in \C$ with $\Re(s) > 1$, we define the \emph{shifted log-sine integral} as the following (normalized) Riemann--Liouville integral of order $2$:
	\[
		\mathrm{SLs}(s; \sigma) \coloneqq \frac{1}{\Gamma(s-1)} \int_0^\sigma (\theta - \sigma) \left(- 2\log \frac{\sin(\theta/2)}{\sin(\sigma/2)}\right)^{s-2}\, \dd \theta.
	\]
\end{definition}

The relationship between log-sine integrals and zeta values has a long history. Early examples include Euler's evaluation~\cite{Euler1773}
\[
	\left(1 - \frac{1}{2^3}\right) \zeta(3) = \frac{\pi^2}{4} \log 2 + 2 \int_0^{\pi/2} \theta \log(\sin \theta)\, \dd \theta,
\]
and Clausen's identity~\cite{Clausen1832}, 
\[
	\sum_{n=1}^\infty \frac{\sin(nx)}{n^2} = -x \log 2 - \int_0^x \log \left(\sin \frac{\theta}{2}\right)\, \dd \theta.
\]
In 1946, a Japanese mathematician Zuiman Y\^{u}j\^{o}b\^{o}~\cite{Yujobo1946} obtained the following identities for integers $n \ge 1$:
\begin{align*}
	\zeta(2n+1) &= \frac{(-1)^n (2\pi)^{2n-1}}{(2n)!} \int_0^{2\pi} B_{2n} \left(\frac{\theta}{2\pi}\right) \log \left(2 \sin \frac{\theta}{2}\right)\, \dd \theta,\\
	\zeta(1,2n) &\coloneqq \sum_{0 < m_1 < m_2} \frac{1}{m_1 m_2^{2n}}\\
		&= \frac{(-1)^n (2\pi)^{2n-1}}{2(2n-1)!} \int_0^{2\pi} B_1\left(\frac{\theta}{2\pi}\right) B_{2n-1} \left(\frac{\theta}{2\pi}\right) \log \left(2 \sin \frac{\theta}{2} \right)\, \dd \theta,
\end{align*}
where $B_n(x)$ is the Bernoulli polynomial. More recently, Borwein--Broadhurst--Kamnitzer~\cite{BBK2001} studied central binomial series and their relations to multiple Clausen values and zeta values.

For integers $s \ge 2$, the values of $\mathrm{SLs}(s; \sigma)$ are related to multiple zeta values and multiple polylogarithms through the general theory of log-sine integrals. In particular, Umezawa~\cite[Theorem~1.4]{Umezawa2023} gave explicit expressions for more general iterated log-sine integrals in terms of these values. For example, a direct calculation gives
\begin{align*}
	\mathrm{SLs}(2; \sigma) &= -\frac{\sigma^2}{2},\\
	\mathrm{SLs}(3; \sigma) &= -2 \left(\zeta(3) - \Re \mathrm{Li}_3(e^{i\sigma}) + \frac{\sigma^2}{2} \log \left(2 \sin \frac{\sigma}{2}\right) \right).
\end{align*}

We now turn to the connection with poly-Bernoulli numbers at negative integers. To describe these values, we first use Lehmer's polynomials $p_n(x)$, which have integer coefficients and are defined in \eqref{def:pq}. The first few examples are
\begin{align*}
	p_{-1}(x) = 0, \quad p_0(x) = 1, \quad p_1(x) = 3, \quad p_2(x) = 8x+7, \quad p_3(x) = 20x^2+70x+15, \dots.
\end{align*}

\begin{theorem}\label{thm:main}
	For $0 < \sigma < \pi$, the shifted log-sine integral $\mathrm{SLs}(s; \sigma)$ extends to an entire function of $s \in \C$. For any integer $n \ge -1$, we have
	\[
		\mathrm{SLs}(-n; \sigma) = \tan^2(\sigma/2)\, \frac{p_n(\sin^2(\sigma/2))}{(2\cos^2(\sigma/2))^n}.
	\]
\end{theorem}

At $\sigma = \pi/3$, this gives
\begin{align}\label{eq:Stephan}
	\mathrm{SLs}(-n; \pi/3) = \frac{1}{3} \left(\frac{2}{3}\right)^n p_n \left(\frac{1}{4}\right) = \frac{1}{3} \sum_{k=0}^n B_{n-k}^{(-k)} \qquad (n \ge 0),
\end{align}
where the second equality is the identity observed by Stephan and proved in~\cite[Theorem~3.6]{BenyiMatsusaka2023}.

More generally, in \cref{sec-4} we introduce a polynomial generalization $\mathcal{B}_n^{(k)}(x)$ of the poly-Bernoulli numbers satisfying $\mathcal{B}_n^{(k)}(2/3) = B_n^{(k)}$. These polynomials give the following extension.

\begin{theorem}\label{thm:Stephan-x}
	For $0 < \sigma < \pi$ and every integer $n \ge 0$, we have
	\[
		\mathrm{SLs}(-n; \sigma) = \tan^2(\sigma/2) \sum_{k=0}^n \mathcal{B}_{n-k}^{(-k)} \left(\frac{1}{2\cos^2(\sigma/2)}\right).
	\]
	The specialization $\sigma = \pi/3$ recovers the preceding formula for the poly-Bernoulli numbers.
\end{theorem}

In \cref{sec-2}, we relate $\mathrm{SLs}(s; \sigma)$ to the central binomial series and its Hurwitz-type generalization, obtaining the entire continuation. In \cref{sec-3}, we use Lehmer's polynomials to prove \cref{thm:main}. The argument in \cref{sec-4} also gives an alternative proof of \eqref{eq:Stephan}.

\section{Central binomial series}\label{sec-2}

A key ingredient in the analytic continuation of $\mathrm{SLs}(s; \sigma)$ is its relation to the central binomial series, which converges for all $s \in \C$. The \emph{central binomial series} is defined by
\begin{align}
	\zeta_\mathrm{CB}(s; z) \coloneqq \sum_{m=1}^\infty \frac{(2z)^{2m}}{\binom{2m}{m} m^s}
\end{align}
for $|z| < 1$ and $s \in \C$. 

\begin{lemma}\label{lem:zcb-int}
	For $0 < \sigma < \pi$ and $\Re(s) > 1$, we have
	\[
		\zeta_\mathrm{CB}(s; \sin(\sigma/2)) = \frac{1}{\Gamma(s-1)} \int_0^\sigma \theta \left(-2\log \frac{\sin(\theta/2)}{\sin(\sigma/2)}\right)^{s-2}\, \dd \theta.
	\]
\end{lemma}

\begin{proof}
	First, we recall that for $|z| < 1$ the function $\arcsin z$ admits the expansion
	\begin{align}\label{eq:arcsin-exp}
		\frac{z \arcsin z}{\sqrt{1-z^2}} = \sum_{m=1}^\infty \frac{(2z)^{2m}}{\binom{2m}{m} 2m}.
	\end{align}
	For simplicity, we put $z = \sin(\sigma/2)$ in this proof. By changing variables via $\sin(\theta/2) = xz$, the right-hand side of the claim becomes
	\begin{align*}
		\frac{1}{\Gamma(s-1)} \int_0^1 (-2\log x)^{s-2} \frac{4z \arcsin (xz)}{\sqrt{1-x^2 z^2}}\, \dd x &= \frac{4}{\Gamma(s-1)} \int_0^1 \frac{(-2\log x)^{s-2}}{x} \sum_{m=1}^\infty \frac{(2xz)^{2m}}{\binom{2m}{m} 2m}\, \dd x\\
			&= \frac{4}{\Gamma(s)} \sum_{m=1}^\infty \frac{(2z)^{2m}}{\binom{2m}{m}} \int_0^1 (s-1) \frac{(-2\log x)^{s-2}}{x} \frac{x^{2m}}{2m}\, \dd x.
	\end{align*}
	Applying integration by parts, this expression becomes
	\begin{align*}
		&= \sum_{m=1}^\infty \frac{(2z)^{2m}}{\binom{2m}{m}} \frac{2}{\Gamma(s)} \int_0^1 (-2\log x)^{s-1} x^{2m-1}\, \dd x = \sum_{m=1}^\infty \frac{(2z)^{2m}}{\binom{2m}{m}} \frac{1}{m^s},
	\end{align*}
	which coincides with the left-hand side of the claim.
\end{proof}

Next, as a half-integer shift analogue of the central binomial series, we introduce the following function for $|z| < 1$ and $s \in \C$:
\begin{align}
	\eta_\mathrm{CB}(s; z) \coloneqq \sum_{m=0}^\infty \frac{(2z)^{2m+1}}{\binom{2m+1}{m+1/2} (m+1/2)^s},
\end{align}
where $\binom{2m+1}{m+1/2} = \Gamma(2m+2)/\Gamma(m+3/2)^2$. This coincides with the case $a=1/2$ of the Hurwitz--Lerch type central binomial series $\Phi_\mathrm{CB}(s,a,z)$ introduced by Ikeda--Kadono~\cite{IkedaKadono2026}. By an argument similar to the one above, we obtain the following log-sine integral representation.

\begin{lemma}
	For $0 < \sigma < \pi$ and $\Re(s) > 1$, we have
	\[
		\eta_\mathrm{CB}(s; \sin(\sigma/2)) = \frac{\pi}{\Gamma(s-1)} \int_0^{\sigma} \left(-2\log \frac{\sin(\theta/2)}{\sin(\sigma/2)} \right)^{s-2}\, \dd \theta.
	\]
\end{lemma}

\begin{proof}
	The proof is obtained by repeating the argument of \cref{lem:zcb-int}, replacing equation \eqref{eq:arcsin-exp} with the identity
	\begin{align}\label{eq:1sq-exp}
		\frac{\pi z}{2\sqrt{1-z^2}} = \sum_{m=0}^\infty \frac{(2z)^{2m+1}}{\binom{2m+1}{m+1/2} (2m+1)}.
	\end{align}
	Hence, we omit the details.
\end{proof}

Therefore, we have the following.

\begin{theorem}\label{thm:SLs-AC}
	For $0 < \sigma < \pi$ and $\Re(s) > 1$, we have
	\[
		\mathrm{SLs}(s; \sigma) = \zeta_\mathrm{CB}(s; \sin(\sigma/2)) - \frac{\sigma}{\pi} \eta_\mathrm{CB}(s; \sin(\sigma/2)).
	\]
	In particular, the expression on the right-hand side provides an analytic continuation to $s \in \C$.
\end{theorem}

\section{Lehmer's polynomials}\label{sec-3}

We define the sequences of polynomials $(p_n(x))_{n=-1}^\infty$ and $(q_n(x))_{n=-1}^\infty$ by the initial conditions $p_{-1}(x) = 0$ and $q_{-1}(x) = 1$, together with the recurrence relations for $n \ge -1$:
\begin{align}\label{def:pq}
\begin{split}
	p_{n+1}(x) &= 2(nx+1) p_n(x) + 2x(1-x) p'_n(x) + q_n(x),\\
	q_{n+1}(x) &= (2(n+1)x+1)q_n(x) + 2x(1-x) q'_n(x).
\end{split}
\end{align}
Then Lehmer~\cite{Lehmer1985} proved the following result.

\begin{proposition}
	For an integer $n \ge -1$, we have
	\[
		\zeta_\mathrm{CB}(-n; z) = \frac{z}{2^n (1-z^2)^{n+3/2}} \bigg(z \sqrt{1-z^2} \cdot p_n(z^2) + \arcsin z \cdot q_n(z^2) \bigg).
	\]
\end{proposition}

Ikeda--Kadono obtained a similar result for $\eta_\mathrm{CB}(-n; z)$. However, in this special case, a more concise statement than their \cite[Theorem 3.3]{IkedaKadono2026} can be derived as follows.

\begin{proposition}
	For an integer $n \ge -1$, we have
	\[
		\eta_\mathrm{CB}(-n; z) = \frac{\pi}{2} \frac{z q_n(z^2)}{2^n (1-z^2)^{n+3/2}}.
	\]
\end{proposition}

\begin{proof}
	The generating function of $\eta_\mathrm{CB}(-n; z)$ is calculated as
	\begin{align}\label{eq:eta-CB-egf}
	\begin{split}
		\sum_{n=0}^\infty \eta_\mathrm{CB}(1-n; z) \frac{t^n}{n!} &= \sum_{n=0}^\infty \sum_{m=0}^\infty \frac{(m+1/2)^{n-1} (2z)^{2m+1}}{\binom{2m+1}{m+1/2}} \frac{t^n}{n!}\\
			&= \sum_{m=0}^\infty \frac{(2z)^{2m+1}}{\binom{2m+1}{m+1/2}} \frac{e^{(m+1/2)t}}{m+1/2}\\
			&= \frac{\pi ze^{t/2}}{\sqrt{1-z^2e^t}}.
	\end{split}
	\end{align}
	The last equality follows from \eqref{eq:1sq-exp}. By \cite[Theorem 2.4]{BenyiMatsusaka2023}, this equals
	\[
		\frac{\pi}{2} \sum_{n=0}^\infty \frac{z q_{n-1}(z^2)}{2^{n-1} (1-z^2)^{n+1/2}} \frac{t^n}{n!},
	\]
	which concludes the proof.
\end{proof}

\begin{proof}[Proof of \cref{thm:main}]
Combining the above propositions with \cref{thm:SLs-AC}, we can eliminate the contribution of the polynomial $q_n(x)$, and obtain
\begin{align*}
	\mathrm{SLs}(-n; \sigma) &= \zeta_\mathrm{CB}(-n; \sin(\sigma/2)) - \frac{\sigma}{\pi} \eta_\mathrm{CB}(-n; \sin(\sigma/2))\\
		&= \frac{z^2}{2^n (1-z^2)^{n+1}} p_n(z^2),
\end{align*}
where $z = \sin(\sigma/2)$.
\end{proof}

The same calculation as in \eqref{eq:eta-CB-egf}, using \eqref{eq:arcsin-exp}, gives
\[
	\sum_{n=0}^\infty \zeta_\mathrm{CB}(1-n;z)\frac{t^n}{n!}
	=\frac{2ze^{t/2}\arcsin(ze^{t/2})}{\sqrt{1-z^2e^t}}.
\]
For $0<\sigma<\pi$ and $z=\sin(\sigma/2)$, combining these generating functions with \cref{thm:SLs-AC} yields
\begin{align}\label{eq:SLs-egf}
	\sum_{n=0}^\infty \mathrm{SLs}(-n;\sigma)\frac{t^{n+1}}{(n+1)!}
	=\frac{ze^{t/2}}{\sqrt{1-z^2e^t}}
	\bigl(2\arcsin(ze^{t/2})-\sigma\bigr).
\end{align}

\section{A polynomial generalization of poly-Bernoulli numbers}\label{sec-4}

We first define auxiliary polynomials $R_n(x)$ by $R_0(x) = 1$ and 
\begin{align}\label{eq:poly-R}
	R_n(x)=n!+\sum_{j=0}^{n-1} \frac{(n+j)!}{(2j+1)!} \bigg((2n+1)x-n-j-1\bigg)R_j(x) \qquad (n \ge 1).
\end{align}
The first few examples are 
\[
	R_1(x) = 3x-1, \quad R_2(x) = 15x^2 - 7x, \quad R_3(x) = 105x^3 - 55x^2 -4x + 2, \dots.
\]

\begin{lemma}\label{lem:R-special}
	For every integer $n \ge 0$, we have
	\[
		R_n(1)=2^n n!, \qquad R_n(2/3)=n!.
	\]
\end{lemma}

\begin{proof}
	The case $n=0$ is clear. For $n\ge1$, put $a_j=(n+j)!j!/(2j)!$. The induction hypotheses and \eqref{eq:poly-R} give
	\begin{align*}
		R_n(1) &=n!+\sum_{j=0}^{n-1} (2^{j+1}a_{j+1}-2^j a_j)=2^n n!,\\
		R_n(2/3) &=n!+\frac23\sum_{j=0}^{n-1} (a_{j+1}-a_j)=n!.
	\end{align*}
	Here we used $a_0=a_n=n!$ and $a_{j+1}/a_j=(n+j+1)/(2(2j+1))$.
\end{proof}

Replacing the factor $j!$ in Kaneko's Stirling-number formula for poly-Bernoulli numbers~\cite[Theorem 1]{Kaneko1997} by $R_j(x)$ leads to the following definition.

\begin{definition}\label{eq:poly-B}
	For $n\ge0$ and $k\in\Z$, we define
	\begin{align*}
	\mathcal{B}_n^{(k)}(x) \coloneqq\sum_{j=0}^n(-1)^{n-j}\sts{n}{j} \frac{R_j(x)}{(j+1)^k},
\end{align*}
where $\sts{n}{j}$ denotes a Stirling number of the second kind.
\end{definition}
By construction and \cref{lem:R-special}, we have $\mathcal{B}_n^{(k)}(2/3)=B_n^{(k)}$. Moreover, $2^{-n} \mathcal{B}_n^{(k)}(1)$ coincides with the poly-Bernoulli number $B_{n,q}^{(k)}$ with parameter $q=1/2$ introduced by Cenkci--Komatsu~\cite[Theorem~2.1]{CenkciKomatsu2015}. This specialization $x=1$ corresponds to $\sigma = \pi/2$ in \cref{thm:Stephan-x}. The first few polynomials $\mathcal{B}_n^{(-k)}(x)$ are as follows:
\[
\begin{array}{c|cccc}
    k\backslash n & 0 & 1 & 2 & 3 \\ \hline
    0 & 1 & 3x-1
      & 15x^2-10x+1
      & 105x^3-100x^2+20x+1 \\[2pt]
    1 & 1 & 6x-2
      & 45x^2-27x+2
      & 420x^3-355x^2+53x+6 \\[2pt]
    2 & 1 & 12x-4
      & 135x^2-75x+4
      & 1680x^3-1285x^2+137x+28 \\[2pt]
    3 & 1 & 24x-8
      & 405x^2-213x+8
      & 6720x^3-4735x^2+335x+120
\end{array}
\]

\begin{proof}[Proof of \cref{thm:Stephan-x}]
Put $z = \sin(\sigma/2)$ and $x = 1/(2(1-z^2))$. For $u = 1 - e^{-t}$, define
\[
	F(u) \coloneqq \sum_{n =0}^\infty \left(\sum_{k=0}^n\mathcal{B}_{n-k}^{(-k)}(x)\right) \frac{t^{n+1}}{(n+1)!}.
\]
We use the following identities (see \cite[Proposition~2.6 (7)]{AIK2014} for the second one):
\begin{align}\label{eq:beta-funct}
	\frac{t^{m+k+1}}{(m+k+1)!} &=\int_0^t \frac{(t-s)^m}{m!} \frac{s^k}{k!} \,\dd s, \\
	\sum_{m = j}^\infty (-1)^{m-j}\sts{m}{j}\frac{s^m}{m!}&=\frac{(1-e^{-s})^j}{j!}.
\end{align}
Writing $n=m+k$ and using \cref{eq:poly-B}, we can sum over $m$ and $k$
separately to obtain
\begin{align*}
	F(u) &=\sum_{j=0}^\infty \frac{R_j(x)}{j!} \int_0^t (1-e^{s-t})^j e^{(j+1)s}\,\dd s\\
		&=\sum_{j=0}^\infty \frac{R_j(x)}{j!} \int_0^u \frac{(u-v)^j}{(1-v)^{2j+2}}\,\dd v,
\end{align*}
where the last equality follows from $v=1-e^{-s}$ and $1 - e^{s-t} = (u-v)/(1-v)$. Expanding $(1-v)^{-2j-2}$ and using \eqref{eq:beta-funct}, we obtain
\begin{align*}
	\frac1{j!}\int_0^u\frac{(u-v)^j}{(1-v)^{2j+2}}\,\dd v &=\sum_{r=0}^\infty\binom{2j+r+1}{r} \frac{r!\,u^{j+r+1}}{(j+r+1)!} =\sum_{n=j}^\infty\frac{(n+j+1)!}{(2j+1)!} \frac{u^{n+1}}{(n+1)!},
\end{align*}
where $n=j+r$ in the last equality. Consequently,
\[
	F(u) = \sum_{n=0}^\infty f_n \frac{u^{n+1}}{(n+1)!}, \qquad f_n =\sum_{j=0}^n\frac{(n+j+1)!}{(2j+1)!}R_j(x).
\]
We have $f_0 =R_0(x)=1$. For $n\ge1$, separating the term $j=n$ and using \eqref{eq:poly-R}, we obtain 
\begin{align*}
	f_n -(2n+1)xf_{n-1} &=R_n(x)-\sum_{j=0}^{n-1}\frac{(n+j)!}{(2j+1)!} \bigg((2n+1)x-n-j-1\bigg)R_j(x) =n!.
\end{align*}
Hence 
\begin{align}\label{eq:poly-F-ode}
	\frac{\partial}{\partial u} \bigg(\sqrt{1-2xu}\,F(u)\bigg) =\frac{1}{(1-u)\sqrt{1-2xu}},\qquad F(0)=0.
\end{align}

Using $\dd u/\dd t = e^{-t}$ and $1- 2xu = (e^{-t}-z^2)/(1-z^2)$, we rewrite \eqref{eq:poly-F-ode} as
\[
	\frac{\partial}{\partial t} \bigg(e^{-t/2} \sqrt{1-z^2 e^t} F(1-e^{-t}) \bigg) = \frac{(1-z^2)e^{t/2}}{\sqrt{1-z^2 e^t}} = \frac{1-z^2}{z} \frac{\partial}{\partial t}\bigg( 2 \arcsin(z e^{t/2}) \bigg).
\]
Integrating from $0$ to $t$ and using $F(0) = 0$ and $2\arcsin z = \sigma$, we obtain
\begin{align*}
	\frac{z^2}{1-z^2} F(1-e^{-t}) &= \frac{z e^{t/2}}{\sqrt{1-z^2 e^t}} \bigg(2 \arcsin(z e^{t/2}) - \sigma \bigg)\\
		&=\sum_{n=0}^\infty\mathrm{SLs}(-n;\sigma) \frac{t^{n+1}}{(n+1)!},
\end{align*}
where the last equality follows from \eqref{eq:SLs-egf}. Comparing coefficients and using $z^2/(1-z^2) = \tan^2(\sigma/2)$ yield
\[
	\mathrm{SLs}(-n; \sigma) = \tan^2 (\sigma/2) \sum_{k=0}^n \mathcal{B}_{n-k}^{(-k)}\left(\frac{1}{2\cos^2 (\sigma/2)} \right),
\]
which concludes the proof.
\end{proof}

\section*{Acknowledgements}

The author is grateful to Masanobu Kaneko for informing him about the paper of Y\^{u}j\^{o}b\^{o}~\cite{Yujobo1946}, which Kaneko learned about from Genki Shibukawa. The author also thanks Hirofumi Tsumura and Ryota Umezawa for several helpful comments on an earlier version of this manuscript. The main ideas of this article were inspired by a seminar talk by Karin Ikeda. The author was supported by JSPS KAKENHI (JP21K18141 and JP24K16901) and by access to OpenAI models provided through the ChatGPT for Academic Researchers program.

\bibliographystyle{amsalpha}
\bibliography{References} 

\end{document}